\documentclass[12]{amsart}
\usepackage[all,2cell,ps]{xy}
\usepackage{amssymb}
\usepackage{amsmath}
\usepackage{amsfonts}
\usepackage[margin=1.2in]{geometry}
\usepackage{geometry}
\usepackage[all,2cell,ps]{xy}

\def\latex/{{\protect\LaTeX}}
\def\latexe/{{\protect\LaTeXe}}
\def\amslatex/{{\protect\AmS-\protect\LaTeX}}
\def\tex/{{\protect\TeX}}
\def\amstex/{{\protect\AmS-\protect\TeX}}
\def\bibtex/{{Bib\protect\TeX}}
\def\makeindx/{\textit{MakeIndex}}

\usepackage{amsmath}
\usepackage{amsthm}
\usepackage{amssymb}
\usepackage{amscd}
\usepackage{amsxtra}     
\usepackage{epsfig}
\usepackage{verbatim}
\usepackage[all]{xypic}

\theoremstyle{plain} 

\newtheorem{thm}{Theorem}[section]
\newtheorem{lem}[thm]{Lemma}
\newtheorem{prop}[thm]{Proposition}

\newtheorem{cor}[thm]{Corollary}

\theoremstyle{definition}
\newtheorem{dfn}[thm]{Definition}

\newtheorem{eg}[thm]{Example}

\newtheorem{rmk}[thm]{Remark}

\newcommand{\ZZ}{\mathbb{Z}}
\newcommand{\QQ}{\mathbb{Q}}

\newcommand{\tensor}{\otimes}
\DeclareMathOperator{\id}{id}

 \DeclareMathOperator{\Tor}{Tor}
\DeclareMathOperator{\Ext}{Ext}
\DeclareMathOperator{\Hom}{Hom}
\DeclareMathOperator{\CI}{\textnormal{CI-dim}}
\DeclareMathOperator{\f}{f^{R}_{ext}}
\DeclareMathOperator{\len}{\lambda}

 \DeclareMathOperator{\pd}{pd}
 
 \DeclareMathOperator{\cx}{cx}
 \DeclareMathOperator{\px}{px}

 \DeclareMathOperator{\depth}{depth}
 
 \DeclareMathOperator{\Cl}{Cl}

 \DeclareMathOperator{\tr}{tor}

 \DeclareMathOperator{\h}{h}

\newcommand{\Ann}{\textup{Ann}}

\def\urltilda{\kern -.15em\lower .7ex\hbox{\~{}}\kern .04em}\def\urldot{\kern -.10em.\kern -.10em}\def\urlhttp{http\kern -.10em\lower -.1ex\hbox{:}\kern -.12em\lower 0ex\hbox{/}\kern -.18em\lower 0ex\hbox{/}}

\begin{document}

\title[Asymptotic Behavior of Ext]{Asymptotic Behavior of Ext functors for modules of finite complete intersection dimension}
\author{Olgur Celikbas and Hailong Dao}

\address{Department of Mathematics, University of Nebraska--Lincoln, 311 Avery Hall,
Lincoln, NE 68588--0323, USA}
\email{s-ocelikb1@math.unl.edu}

\address{Department of Mathematics, University of Kansas, 405 Snow Hall, 1460 Jayhawk
Blvd, Lawrence, KS 66045-7523, USA}
\email{hdao@math.ku.edu}

\thanks{The second author is partially supported by NSF grant DMS 0834050}



\begin{abstract}
Let $R$ be a local ring, and let $M$ and $N$ be finitely generated $R$-modules such that $M$ has finite complete intersection dimension. In this paper we define and study, under certain conditions, a pairing using the modules $\Ext_R^i(M,N)$ which generalizes Buchweitz's notion of the Herbrand diference. We exploit this pairing to examine the number of consecutive vanishing of $\Ext_R^i(M,N)$ needed to ensure that $\Ext_R^i(M,N)=0$ for all $i\gg 0$. Our results recover and improve on most of the known bounds in the literature, especially when $R$ has dimension at most two.

\end{abstract}

\maketitle
\section{Introduction}
Let $R$ be a local ring, and let $M$ and $N$ be finitely generated $R$-modules. An active area of research in commutative algebra has been aimed at understanding the vanishing pattern of the modules $\Ext^i_R(M,N)$ and $\Tor_i^R(M,N)$. One reason for that is one frequently needs to understand the properties of $\Hom_{R}(M,N)$ and $M\otimes_{R}N$, and it naturally leads to investigating the corresponding derived functors. Vanishing of such functors over special rings carries great deal of information; If $R$ is a complete intersection ring and $\Tor_i^R(M,N)=0$ for all $i\geq 1$, then the depth formula holds for $M$ and $N$, i.e., $\depth(M)+\depth(N)=\depth(R)+\depth(M\otimes_{R}N)$ \cite[2.5]{HW}. As another example, it was proved by Avramov and Buchweitz in \cite[Theorem 4.2]{AvBu} that if $R$ is a complete intersection, then $\Ext_R^n(M,M)=0$ for some {\it even} $n$ forces $M$ to have finite projective dimension.

One could argue that the topic started with the elegant rigidity theorem of Auslander and Lichtenbaum \cite[Corollary 2.2]{Au} and \cite[Corollary 1]{Lic}; If $R$ is a regular local ring and $\textnormal{Tor}^{R}_{n}(M,N)=0$ for some nonnegative integer $n$, then $\textnormal{Tor}^{R}_{i}(M,N)=0$ for all $i\geq n$. Murthy \cite[1.6]{Mu} exploited the rigidity result and proved that, over a complete intersection of codimension $r$, $r+1$ consecutive vanishing $\Tor$ modules forces the vanishing of all subsequent $\Tor$s. Various generalizations of Murthy's result have been obtained in the literature. These results mostly focus on the following themes: reducing the number of initial vanishing necessary \cite{Ce, Da3, Da2, HW, Jo1}, studying $\Ext$ instead of $\Tor$ \cite{AvBu, Be1, Jo2}, investigating asymptotic behaviour \cite{Be1, BeJo, Da2, Jo1, Jo3}, and recently, replacing the vanishing assumption of consecutive indexes to that of an arithmetic progression \cite{Be1, BeJo}.

In this paper we study the vanishing pattern of $\Ext^i_R(M,N)$ when $M$ has finite complete intersection dimension, a situation which is slightly general than assuming $R$ is a complete intersection (see Section \ref{Notations} for the definition of finite complete intersection dimension). This assumption on $M$ is quite standard; Modules of finite complete intersection dimension behave homologically like modules over complete intersections \cite{AGP}. On the other hand Murthy's theorem does not extend to rings that are not complete intersections \cite[3.2]{Jo2}, so clearly some assumptions are needed.

One of the main technical tools in this paper is a generalization of \emph{Herbrand difference} that was introduced by Buchweitz in \cite{Bu} (see Definitions \ref{HerBu} and \ref{function}) for maximal Cohen-Macaulay modules over an isolated hypersurface singularity. The Grothendieck group of finitely generated $R$-modules is an additional ingredient  used consistently. Our approach yields sharper results than most of the bounds for vanishing pattern of $\Ext_R^i(M,N)$ previously known. It works especially well in low dimensions where the behaviour of the Grothendieck group is better understood.

As an example, we can state:

\begin{thm} \label{A}
Let $R$ be a local ring, and let $M$ and $N$ be finitely generated $R$-modules. Assume $M$ has finite complete intersection dimension and $R$ satisfies one of the following:
\begin{enumerate}
\item $R$ is Artinian.
\item $R$ is a one-dimensional domain.
\item $R$ is a two-dimensional normal domain with torsion class group.
\end{enumerate}
If $\Ext^{n}_{R}(M,N)=\dots=\Ext^{n+c-1}_{R}(M,N)=0$
for some $n>\depth(R)-\depth(M)$, where $c=\cx_{R}(M)$, the complexity of $M$, then $\Ext^{i}_{R}(M,N)=0$ for all $i>\depth(R)-\depth(M)$.
\end{thm}

Theorem \ref{A} (see Corollaries \ref{cor1} and \ref{alfa8}) reduces the required number of vanishing $\Ext$ modules from $c+1$ to $c$ in some cases (cf. \cite[2.6(1)]{Jo2}). We even improve the above result in Proposition \ref{alfa9} over one-dimensional Gorenstein domains.

In the general case, if there are $c+1$ consecutive vanishing $\Ext$ modules, then we can replace $c$ with the complexity of the pair $(M,N)$, $\cx_{R}(M,N)$, a nonnegative integer that is at most $\cx_{R}(M)$. More precisely, we prove the following result as Corollary \ref{alfa3}:

\begin{thm}\label{B}
Let $R$ be a local ring, and let $M$ and $N$ be finitely generated $R$-modules. Assume $M$ has finite complete intersection dimension. If $\Ext^{n}_{R}(M,N)=\dots=\Ext^{n+c}_{R}(M,N)=0$ for some $n>\depth(R)-\depth(M)$, where $c=\cx_{R}(M,N)$, then $\Ext^{i}_{R}(M,N)=0$ for all $i>\depth(R)-\depth(M)$.
\end{thm}

Various results of similar flavours can be found in Section \ref{mainSection}.

We now describe the organization of the paper. Section \ref{Notations} contains notations and some preparatory results. In Section \ref{eta} we develop some of the general results about the asymptotic behaviour of $\Ext^i_R(M,N)$ under the extra condition that they eventually have finite length. This section culminates with a definition of the generalized Herbrand difference and an investigation of its basic properties. In Section \ref{mainSection}, we prove our main results on the vanishing pattern of $\Ext^i_R(M,N)$.

\section{Notations and Preliminary Results}\label{Notations}

Let $(R,\mathfrak{m},k)$ be a \textit{local} ring, i.e., a commutative
Noetherian ring with unique maximal ideal $\mathfrak{m}$, and let
$M$ and $N$ be a finitely generated $R$-modules.

The \textit{complexity} of a sequence of non-negative integers $B =
\{b_i\}_{i\geq 0}$ is defined in \cite{DaV} as follows: $$\cx(B) = \inf \ \{r\in \ZZ \mid
b_n\leq A \cdot n^{r-1} \ \text {for some real number} \  A  \ \text
{and for all} \ n\gg 0 \}$$ According to this notation the
complexity of the pair $(M,N)$, introduced in \cite{AvBu}, can be
defined as:
$$\cx_{R}(M,N) = \cx\left( \{ \nu_{R}(\Ext^{i}_{R}(M,N)) \} \right)$$
where $\nu_{R}(X)$ denotes the minimal number of generators of a finitely generated
$R$-module $X$. The complexity $\cx_{R}(M)$ and the \textit{plexity} $\px_{R}(M)$ \cite{Av1} of $M$
are defined as $\cx_{R}(M,k)$ and $\cx_{R}(k,M)$, respectively.

If $\textbf{F}:\ldots \rightarrow F_{2}
\rightarrow F_{1} \rightarrow F_{0}  \rightarrow 0$ is a minimal free
resolution of $M$ over $R$, then the rank of $F_{n}$, i.e., the integer $\dim_{k}(\Ext^{n}_{R}(M,k))$,
is the $n$th \textit{Betti} number of $M$. This integer is well-defined for all $n$ since minimal free resolutions over $R$ are unique up to isomorphism. It follows from the definition of the complexity that $M$ has finite projective dimension if and only if $\cx_{R}(M)=0$, and has bounded Betti numbers if and only if $\cx_{R}(M)\leq 1$.

A \textit{quasi-deformation} of $R$ \cite{AGP} is a
diagram $R \rightarrow S \twoheadleftarrow P$ of local
homomorphisms, where $R\rightarrow S$ is flat and
$S\twoheadleftarrow P$ is surjective with kernel generated by a
regular sequence of $P$ contained in the maximal ideal of $P$. $M$ is said to have finite \textit{complete
intersection dimension}, denoted by $\CI_{R}(M)<\infty$, if there
exists a quasi-deformation $R \rightarrow S \twoheadleftarrow P$
such that $\pd_{P}(M\tensor_{R}S)<\infty$. Modules of finite projective dimension and modules over complete intersection
rings have finite complete intersection dimension. Recall that $R$ is said to be a \textit{complete intersection} if the defining ideal of some (equivalently every) Cohen presentation of the $\mathfrak m$-adic completion $\widehat{R}$ of $R$
can be generated by a regular sequence. If $R$ is such a ring, then $\widehat{R}$ has the form $Q/(\underline{f})$,
where $\underline{f}$ is a regular sequence of $Q$ and $Q$ is a ring of formal power series over the field $k$, or over
a complete discrete valuation ring with residue field $k$. The non-negative integer $\nu_{R}(\mathfrak m)-\dim(R)$ is the \textit{codimension} of $R$. A complete intersection of codimension one is called a \textit{hypersurface}.

If $\CI_{R}(M)<\infty$, then it follows from \cite[4.1.2]{AvBu} and \cite[1.4 \& 5.6]{AGP} that the following (in)equalities hold:
$$\CI_{R}(M)=\depth(R)-\depth(M) \textnormal{ and }
\cx_{R}(M,N)\leq \cx_{R}(M)\leq \nu_{R}(\mathfrak m)-\depth(R).$$

Assume $R=Q/(\underline{f})$ where $Q$ is a local
ring and $\underline{f}=f_{1},f_{2}, \dots, f_{r}$ is a regular
sequence of $Q$. Given a minimal free resolution
$\textbf{F}$ of $M$ over $R$, the regular sequence $\underline{f}$ gives rise to
chain maps, of degree $-2$, on $\textbf{F}$ (cf. \cite{Av1} and \cite{Ei}). These
chain maps are uniquely defined and commute, up to homotopy. They
give rise to a polynomial ring of cohomology operators
$\mathcal{S}=R[\chi_{1}, \chi_{2}, \ldots, \chi_{r}]$ with variables $\chi_{i}$ of
degree two, and the graded $R$-module $\Ext^{\ast}_{R}(M,N)$, that is the direct sum of
$\Ext^{i}_{R}(M,N)$ for all $i\geq 0$, has a graded module structure over $\mathcal{S}$.

If the length of the modules $\Ext^i_R(M,N)$ is finite for all $i\gg 0$, then we denote by $\f(M,N)$ the set $\textnormal{inf}\{s:\len_{R}(\Ext^i_R(M,N)<\infty \textnormal{ for all } i\geq s \}$,
where $\len_{R}(X)$ denotes the length of an $R$-module $X$.
For notational convenience, we shall write $\beta_i^R(M,N) = \len_{R}(\Ext^i_R(M,N))$ for all $i\geq \f(M,N)$.

We now record some of the results that will be used in the next sections.

\begin{thm} \label{Gul} (\cite[4.2]{AGP} and \cite[3.1]{Gu}) Let $R$ be a ring such that $R=Q/(\underline{f})$
where $Q$ is a local ring and $\underline{f}$ is a regular sequence of $Q$, and let $M$ and $N$ be finitely generated $R$-modules.  Let $\mathcal{S}=R[\underline{\chi}]$ be the ring of cohomology operators defined by the regular sequence $\underline{f}$. Then $\Ext^{\ast}_{R}(M,N)$ is a finitely generated graded module over $\mathcal{S}$ if and only if $\Ext^{i}_{Q}(M,N)=0$ for all $i\gg 0$.
\end{thm}

\begin{prop}\label{lencx} Let $R$ be a ring such that $R=Q/(\underline{f})$
where $Q$ is a local ring and $\underline{f}$ is a regular sequence of $Q$.
Let $M$ and $N$ be a finitely generated $R$-modules. Assume $\f(M,N)<\infty$ and $\Ext^{i}_{Q}(M,N)=0$ for all $i\gg 0$.
Then $\cx_{R}(M,N)=\cx\left( \{ \len_{R}(\Ext^{i}_{R}(M,N)) \} \right)$.
\end{prop}

\begin{proof} Theorem \ref{Gul} shows that the graded module $\Ext^{\ast}_{R}(M,N)$ is Noetherian over the cohomology operators $\mathcal{S}$ defined by the regular sequence $\underline{f}$. Thus the sequence of ideals $\{\Ann_{R}(\Ext^{i}_{R}(M,N))\}$ eventually becomes periodic of period two \cite[2.4]{Da2}. Hence, since $\f(M,N)<\infty$, one can find a positive integer $h$ such that $\mathfrak m^{h}\Ext^{i}_{R}(M,N)=0$ for all $i\gg 0$, where $\mathfrak m$ is the unique maximal ideal of $R$. Now the result follows from \cite[2.5]{DaV}.
\end{proof}

\begin{prop}\label{1.2} (\cite[11.65]{J}) \label{longexact}
Let $R=Q/(x)$ where $Q$ is a commutative ring and $x$ is a non-zerodivisor of $Q$.
If $M$ and $N$ are $R$-modules, then one has the change of rings long exact sequence of
$\Ext$:
\begin{equation*}
\begin{split}
0 \to \Ext_R^1(M,N) \to  \Ext_Q^1(M,N) \to \Ext_R^{0}(M,N) \to \dots  \to \Ext_R^n(M,N) \to \Ext_Q^n(M,N) \to \\  \Ext_R^{n-1}(M,N) \to  \Ext_R^{n+1}(M,N) \to \Ext_Q^{n+1}(M,N) \to \Ext_R^{n}(M,N) \to \dots
\end{split}
\end{equation*}
\end{prop}

\begin{lem}\label{factor}
Let $R\twoheadleftarrow P$ be a surjection of local rings, and let $M$ and $N$ be finitely generated $R$-modules. Assume $\pd_P(M)<\infty$ and the kernel of $R\twoheadleftarrow P$ is generated by a regular sequence of $P$. Assume further that the residue field $k$ of $P$ is infinite. Set $c=\cx_R(M,N)$. If $c\geq 1$, then the surjection $R\twoheadleftarrow P$ can be factored as $R\twoheadleftarrow Q\twoheadleftarrow P$ such that $\cx_{Q}(M,N)=0$ and the kernel of $R \twoheadleftarrow Q$ is generated by a regular sequence of $Q$ of length $c$. Furthermore, if $\f(M,N)<\infty$ (i.e., $\len_R(\Ext_R^i(M,N))<\infty$ for all $i\gg0$), then there exits a local ring $R'$ and a non-zerodivisor $x$ of $R'$ such that $R=R'/(x)$, $\CI_{R'}(M)<\infty$ and $\cx_{R'}(M,N) = \cx_R(M,N)-1$.
\end{lem}

\begin{proof}
The proof follows from that  of \cite[9.3.1]{Av1} and \cite[5.9]{AGP}. Note that Theorem \ref{Gul} shows the graded module $\mathcal{E}=\Ext^{\ast}_{R}(M,N)\tensor_{R}k$ is finitely generated over $\mathcal{R}=\mathcal{S}\tensor_{R}k$, where $\mathcal{S}=R[\underline{\chi}]$ is the ring of cohomology operators defined by the deformation $R\twoheadleftarrow P$. Then, as $\mathcal{E}\neq 0$, $\dim_{\mathcal{R}}\mathcal{E}=c\geq 1$ \cite[1.3]{AvBu}.
The kernel of $R\twoheadleftarrow P$ can be generated by a regular sequence $f_{1},f_{2}, \dots, f_{r}$ that defines $\underline{\chi}=\chi_{1},\chi_{2},\dots,\chi_{r}$ such that; (i) the ring of cohomology operators defined by the presentation $R=Q/(f_{1},f_{2}, \dots, f_{c})$, where $Q=P/(f_{c+1},f_{2}, \dots, f_{r})$, is identified with $\mathcal{R'}=k[\chi_{1},\chi_{2},\dots,\chi_{c}] \subseteq \mathcal{R}$ and (ii) $\chi_{1},\chi_{2},\dots,\chi_{c}$ form a system of parameters on $\mathcal{E}$. Since $\mathcal{E}$ is finitely generated over $\mathcal{R'}$, Nakayama's lemma and Theorem \ref{Gul} imply that $\cx_{Q}(M,N)=0$.

Now assume $\f(M,N)<\infty$ and write $R=R'/(x)$ where $R'=Q/(f_{1},\dots, f_{c-1})$ and $x=f_{c}$. Note that, since $\pd_{P}(M)<\infty$, the map $Q\twoheadleftarrow P$ implies that $\CI_{Q}(M)<\infty$. Furthermore $\CI_{R'}(M)+(c-1)\leq \CI_{Q}(M)$ \cite[1.2.3]{AGP}. Thus $\CI_{R'}(M)<\infty$. As $R=R'/(x)$ and $R'=Q/(f_{1},\dots, f_{c-1})$, Proposition \ref{lencx} and the argument in the proof of \cite[1.5]{AGP} yield the following inequalities:
$$\cx_{R}(M,N)-1\leq \cx_{R'}(M,N),\;\; \cx_{R'}(M,N)\leq \cx_{Q}(M,N)+c-1=\cx_{R}(M,N)-1.$$
Therefore $\cx_{R'}(M,N)=\cx_{R}(M,N)-1$.
\end{proof}

We denote by $G(R)$ the \textit{Grothendieck group} of finitely generated $R$-modules, that is, the quotient of the free abelian group of all isomorphism classes of finitely generated $R$-modules by the subgroup generated by the relations coming from short exact sequences of finitely generated $R$-modules. We write $[M]$ for the class of $M$ in $G(R)$ and denote by $\overline{G}(R)$ the group $G(R)/\ZZ \cdot [R]$, the reduced Grothendieck group of $R$. We set $G_{\QQ}=G\tensor_{\ZZ}\QQ$ for an abelian group $G$. Some facts about the group $\overline{G}(R)_{\QQ}$ are collected in the next proposition.

\begin{prop}\label{Groth}
Let $R$ be a local ring, and let $N$ be a finitely generated $R$-module. Then $[N]=0$ in $\overline{G}(R)_{\QQ}$ for each one of the following cases:
\begin{enumerate}
\item $N$ has finite length, or $N$ is a syzygy of some finite length $R$-module.
\item $R$ is Artinian.
\item $R$ is a one-dimensional domain.
\item $R$ is a two-dimensional normal domain with torsion class group.
\end{enumerate}
\end{prop}

\begin{proof}
Let $k$ denote the residue field of $R$. Assume $N$ has finite length. We claim $[N]=0$ in $\overline{G}(R)_{\QQ}$. Note that $[N]=l \cdot [k]$ where $l=\len_{R}(N)$. Therefore it suffices to prove $[X]=0$ in $\overline{G}(R)$ for some finite length $R$-module $X$. If $\dim R=0$, then there is nothing to prove as we kill the class $[R]$ of $R$. Suppose now $\dim R>0$. Choose a prime ideal $p$ and an element $x$ of $R$ such that $x\notin p$ and $\dim(R/p)=1$. Then the short exact sequence $\displaystyle{0 \rightarrow R/p \stackrel{x}{\rightarrow} R/p \rightarrow R/(p+x) \rightarrow 0}$ implies that $[R/(P+(x))] =0$. This proves the claim. Therefore $(1)$ and $(2)$ follow.

Suppose now $R$ is a domain. Then there is an exact sequence $0\rightarrow K \rightarrow  R^{(t)} \rightarrow N \rightarrow C \rightarrow 0$ where $K$ and $C$ are torsion $R$-modules. If $\dim(R)=1$, then $K$ and $C$ have finite length, and hence $[N]=0$ in $\overline{G}(R)_{\QQ}$. This proves $(3)$.

Next assume that $R$ is a two-dimensional normal domain. Then $\overline{G}(R)_{\QQ}\cong\Cl(R)_{\QQ}$ where $\Cl(R)$ is the class group of $R$. As $\Cl(R)$ is torsion, this implies $\overline{G}(R)_{\QQ}=0$ and hence proves $(4)$. For the reader's convenience we will sketch a proof for the isomorphism above: Since $\dim(R)=2$, there is a well-defined map $\alpha:\Cl(R) \to \overline{G}(R)_{\QQ}$ given by $\displaystyle{\alpha\left([p]\right)=[R/p]}$ for height one prime ideals $p$ of $R$. The maps $\gamma$ and $\delta$ in the localization exact sequence \cite[6.2]{Bass} $$F^{\times} \stackrel{\gamma}{\longrightarrow} G(\tr(R)) \longrightarrow G(R) \stackrel{\delta}{\longrightarrow} \ZZ \longrightarrow 0 $$
are defined as $\gamma\left(a/b\right)=[R/a]-[R/b]$ and $\delta\left([M]\right)=\dim_F(M\otimes_{R}F)$. Here $F$ is the field of fractions of $R$ and $G(\tr(R))$ is the Grothendieck group of finitely generated torsion $R$-modules. This shows that $\overline{G}(R)$ is isomorphic to the free abelian group on finitely generated torsion $R$-modules modulo the classes of the form $[R/x]$ where $x\in R-\{0\}$, and the relations coming from short exact sequences of torsion $R$-modules. Hence one has a well-defined map $\beta:\overline{G}(R) \to \Cl(R)$, where $\displaystyle{\beta\left([M]\right)=\sum \lambda_{R_{p}}(M_{p})[p]}$ with the sum is taken over all height one prime ideals $p$ of $R$. Therefore the isomorphism $\overline{G}(R)_{\QQ}\cong\Cl(R)_{\QQ}$ follows from the fact that $\alpha\otimes_{\ZZ}\QQ$ and $\beta\otimes_{\ZZ}\QQ$ are inverses to each other.
\end{proof}

\begin{rmk}\label{Grothrmk} The condition (4) of Proposition \ref{Groth} is quite subtle. Let $R$ be an excellent two-dimensional normal domain. For the applications in this paper, one can make suitable flat extensions to assume that $R$ is complete and the residue field $k$ of $R$ is algebraically closed. Suppose $k$ is of characteristic zero. Then $\Cl(R)$ is torsion if and only if $R$ is a rational singularity; This result was explained to us by S. D. Cutkosky. (cf. \cite[17.4]{Lip} and the corollary after \cite[Theorem 4]{Cu}).

In positive characteristic, by the non-trivial results in \cite[Theorem 4]{Cu} and \cite[4.5]{Go}, the class of rings satisfying (4) include {\it all} two-dimensional complete normal domains such that $k$ is either finite or is the algebraic closure of a finite field.
\end{rmk}

\section{Asymptotic behavior of $\beta_i^R(M,N)$ and the generalized Herbrand function}\label{eta}

In this section we will adapt the arguments of \cite{Da2} and define an asymptotic function associated to $\Ext^{i}_{R}(M,N)$ for a pair of finitely generated $R$-modules $(M,N)$. This function can be viewed as a natural generalization of the notion of Herbrand difference, defined by Buchweitz.

Recall from Section \ref{Notations} that if  the length of the modules $\Ext^i_R(M,N)$ is finite for all $i\gg 0$, then we denote by $\f(M,N)$ the set $\textnormal{inf}\{s:\len_{R}(\Ext^i_R(M,N)<\infty \textnormal{ for all } i\geq s \}$, where $\len_{R}(X)$ denotes the length of an $R$-module $X$.
We also write $\beta_i^R(M,N) = \len_{R}(\Ext^i_R(M,N))$ for all $i\geq \f(M,N)$. We now state a slightly modified version of Buchweitz's definition for possibly non-maximal Cohen-Macaulay modules (for more details, see Section 10.2 of \cite{Bu}).

\begin{dfn}\label{HerBu}(\cite{Bu})
Let $R$ be a local hypersurface with an isolated singularity, i.e., $R_p$ is regular for all non-maximal prime ideals $p$ of $R$. For a pair of finitely generated $R$-modules $(M,N)$, the sequence of modules $\{\Ext_R^i(M,N)\}$ is periodic of period two and has finite length for all $i> \depth R -\depth M$. The Herbrand difference $\h^{R}(M,N)$ of $(M,N)$ is defined as:
$$\h^{R}(M,N) = \beta_n^R(M,N) - \beta_{n-1}^R(M,N)$$
where $n$ is any even number such that $n>\depth R-\depth M+1$.
\end{dfn}

The Herbrand difference is relevant for proving results for the vanishing pattern of $\Ext$ modules because of the following simple observation: Suppose that $\h^{R}(M,N)=0$. If $\Ext_R^n(M,N)=0$ for some $n>\depth R-\depth M+1$, then $\Ext_R^i(M,N)=0$ for all $i\geq n$. To generalize this function, we first prove that the numbers $\beta_i^R(M,N)$ share similar properties as the Betti numbers of the module $M$ (cf. also \cite[9.2.1]{Av1}).

\begin{prop}\label{beta1}
Let $R$ be a ring such that $R = Q/(\underline{f})$ where $Q$ is a local ring and $\underline{f}=f_1,...,f_r$ is a regular sequence of $Q$. Let $M$ and $N$ be finitely generated $R$-modules. Assume $\f(M,N)<\infty$ and $\Ext^i_Q(M,N)=0$ for all $i\gg0$ (which is automatic if $\pd_{Q}(M)<\infty$). Set
$$
\displaystyle{P_{M,N}^R(t) = \sum_{i=\f(M,N)}^{\infty}\beta_i^R(M,N)t^i}
$$

Then the following hold:

\begin{enumerate}
\item There is a polynomial $p(t) \in \mathbb{Z}[t]$ with $p(\pm 1)\neq 0$, such that $\displaystyle{P_{M,N}^R(t)  = \frac{p(t)}{(1-t)^c(1+t)^d}}$.\\
\item $\displaystyle{ \beta_i^R(M,N) = \frac{m_0}{(c-1)!} i^{c-1} + (-1)^i\frac{n_0}{(d-1)!} i^{d-1}+ q_{(-1)^i}(i)}$
for all $i\gg 0$, where $m_0$ is a non-negative rational  number and $g_{\pm}(t)\in \mathbb{Q}[t]$ are polynomials of degrees $< \max\{c,d\}-1$.\\
\item $d \leq c = \cx_R(M,N) \leq r$.\\
\end{enumerate}
\end{prop}

\begin{proof}\hspace{0.01in}
\begin{enumerate}
\item Set $\displaystyle{\xi=\bigoplus^{\infty}_{i=\f(M,N)}\Ext^i_R(M,N)}$. Let $\mathfrak m$ be the unique maximal ideal of $R$. Then Theorem \ref{Gul} and the proof of Proposition \ref{lencx} show that there exists a positive integer $h$ such that $\xi$ is a finitely generated graded module over the ring $\mathcal{T}=\displaystyle{(R/\mathfrak m^{h})[\chi_{1}, \chi_{2}, \ldots, \chi_{r}]}$, where each $\chi_i$ has degree $2$. Therefore Hilbert-Serre Theorem applies to the module $\xi$ over $\mathcal{T}$. This shows that $\displaystyle{P_{M,N}^R(t)=\frac{h(t)}{(1-t^2)^r}}$ for some polynomial $h(t) \in \mathbb{Z}[t]$. Now, cancelling the powers of $1-t$ and $1+t$, one can find a polynomial $p(t) \in \mathbb{Z}[t]$ such that $\displaystyle{P_{M,N}^R(t)  = \frac{p(t)}{(1-t)^c(1+t)^d}}$ where $p(\pm 1)\neq 0$.\\

\item  From part (1), we can write :
$$
\sum_{i=\f(M,N)}^{\infty}\beta_i^R(M,N)t^i= \frac{p(t)}{(1-t)^c(1+t)^d} =\sum_{l=0}^{c-1}{\frac{m_l}{(1-t)^{c-l}}} +
                                         \sum_{l=0}^{d-1}{\frac{n_l}{(1+t)^{d-l}}} + q(t)
$$
Here $q(t) \in \mathbb{Z}[t]$.
Then, by comparing coefficients from both sides, we get the desired formula for $\beta_i^R(M,N)$.
Since $\beta_i^R(M,N) \geq 0$, $m_0$ must be a non-negative rational number.\\
\item
That $c \leq r$ is obvious by the proof of (1). Since the sign of $\beta_i^R(M,N)$ for odd $i$ is positive only if
$c \geq d$, the first inequality is also clear. The size of $\beta_i^R(M,N)$ behaves like a polynomial
of degree $\max\{c,d\}-1= c-1$. Therefore, by Proposition \ref{lencx}, we see that $\cx_R(M,N) = c$.
\end{enumerate}
\end{proof}

\begin{dfn}\label{function}
Let $R$ be a local ring, and let $M$ and $N$ be finitely generated $R$-modules. Assume that $\f(M,N)<\infty$. Then, for a non-negative integer $e$, $\h_e^R(M,N)$ is defined as follows:
\begin{equation}
\begin{split}
 \h_e^R(M,N) = \lim_{n\to\infty} \frac{\displaystyle{\sum\limits^{n}_{i=\f(M,N)}(-1)^i \beta_i^R(M,N)}}{\displaystyle{n^e}}
\end{split}\notag{}
\end{equation}
\end{dfn}

The behavior of $\beta_i^R(M,N)$, proved in Proposition \ref{beta1}, shows that the function $\h^{R}_{\bullet}(M,N)$ behaves quite well:

\begin{thm}\label{beta2}
Let $R$ be a ring such that $R = Q/(\underline{f})$ where $Q$ is a local ring and
$\underline{f}=f_1,...,f_r$ is a regular sequence of $Q$. Let $M$ and $N$ be finitely generated $R$-modules. Assume $\f(M,N)<\infty$ and $\Ext^i_Q(M,N)=0$ for all $i\gg 0$. Set $c = \cx_R(M,N)$.

\begin{enumerate}
\vspace{0.1in}

\item If $e$ is an integer such that $e\geq c$, then $\h_e^R(M,N)$ is finite. Moreover, if $e>c$, then $\h_e^R(M,N)=0$.
\vspace{0.1in}

\item (Biadditivity)
\vspace{0.1in}

\noindent (i) Let $\displaystyle{0 \to M_1 \to M_2 \to M_3 \to 0}$ be an exact sequences of finitely generated $R$-modules. Assume $\f(M_j,N)<\infty$ and $\Ext^i_Q(M_j,N)=0$ for all $i\gg 0$ and for all $j$. Assume further that $e$ is an integer such that $e\geq \cx_{R}(M_j,N)$ for all $j$. If $e \geq 1$, then $$\h_e^R(M_2,N) = \h_e^R(M_1,N) + \h_e^R(M_3,N).$$ Moreover, if $e=0$ and $\len_{R}(M\tensor_RN)<\infty$, then $$\h_0^R(M_2,N) = \h_0^R(M_1,N) + \h_0^R(M_3,N).$$
(ii) Let $\displaystyle{0 \to N_1 \to N_2 \to N_3 \to 0}$ be an exact sequence of finitely generated $R$-modules. Assume $\f(M,N_{j})<\infty$ and $\Ext^i_Q(M,N_{j})=0$ for all $i\gg 0$ and for all $j$. Assume further that $e$ is an integer such that $e\geq \cx_{R}(M,N_{j})$ for all $j$. If $e \geq 1$, then $$\h_e^R(M,N_{2}) = \h_e^R(M,N_{1}) + \h_e^R(M,N_{3}).$$ Moreover, if $e=0$ and $\len_{R}(M\tensor_RN)<\infty$, then $$\h_0^R(M,N_{2}) = \h_0^R(M,N_{1}) + \h_0^R(M,N_{3}).$$

\item (Change of rings)
\vspace{0.1in}

\noindent Suppose that $r\geq 1$ and set $R'=Q/(f_1,...,f_{r-1})$. Let $e$ be a positive integer such that $e\geq c$. Assume $\cx_{R'}(M,N)\leq e-1$.
If $e\ge 2$, or $e=1$ and $\len_{R}(M\tensor_RN)<\infty$, then $\displaystyle{ 2 \cdot \h_e^{R}(M,N) = \h_{e-1}^{R'}(M,N)}$.\\
\end{enumerate}
\end{thm}

\begin{proof}
Let $n$ and $h$ be integers such that $n>h$. Set $g_{M,N}^R(h,n) = \sum_{i=n}^{h} (-1)^i \beta_i^R(M,N))$. Assume $e\geq 1$. Then, for
a fixed $h$, it is clear that:
$$ \h_{e}^R(M,N) = \lim_{n\to\infty} \frac{g_{M,N}^R(h,n)}{n^e}$$
(1)
If $e=0$, then $c=0$ and hence $\Ext^i_R(M,N)=0$ for all $i\gg0$. Thus $\h_e^{R}(M,N)$ is finite. Assume now $e\geq 1$. We choose a sufficiently large integer $h$ so that the formula for $\beta_i^R(M,N)$ in Proposition \ref{beta1}(2) is true for all $i \ge h$. Then,
\begin{equation}\label{eq1}
\begin{split}
g_{M,N}^R(h,n) & = \sum_{i=h}^{n} (-1)^i \beta_i^R(M,N)\\
& = \frac{m_0}{(c-1)!}\sum_{i=h}^{n}(-1)^i i^{c-1}
                     + \frac{n_0}{(d-1)!}\sum_{i=h}^{n}i^{d-1}
                     +\sum_{i=h}^{n}(-1)^i g_{(-1)^i}(i)
\end{split}\tag{\ref{beta2}.1}
\end{equation}
Note that $\displaystyle{\sum_{i=h}^{n}(-1)^i i^{c-1}}$ and $\displaystyle{\sum_{i=h}^{n}(-1)^i g_{(-1)^i}(i)}$ are polynomials in $n$ of order $c-1$ and at most $c-2$, respectively. Since $e\geq c$, it follows from (\ref{eq1}) that:

\begin{equation}\label{eq2}
\h_{e}^R(M,N) =\lim_{n\to\infty} \frac{g_{M,N}^R(h,n)}{n^e}= \lim_{n\to\infty} \frac{n_0}{(d-1)!} \cdot \frac{\sum_{i=h}^{n}i^{d-1}}{n^e}
\tag{\ref{beta2}.2}
\end{equation}
Using (\ref{eq2}) and the equality $\displaystyle{ \sum_{i=h}^{n} i^{d-1}=\frac{n^{d}}{d}+}$ lower order terms, we have:
\begin{equation}\label{eq3}
\h_{e}^R(M,N) =\lim_{n\to\infty} \frac{n_0}{d!}n^{d-e}
\tag{\ref{beta2}.3}
\end{equation}
The claim now follows from the fact that $d-e$ is a non-negative integer (see Proposition \ref{beta1}(3)).\\\\
(2)
It is enough to prove (i) since (ii) follows in an identical manner. Assume $e\geq \cx_{R}(M_j,N)$ for each $j$. The short exact sequence $0 \to M_1 \to M_2 \to M_3 \to 0$ gives rise to the following long exact sequence
\begin{equation}\label{eq4}
\dots \to \Ext^i(M_3,N) \to \Ext^i(M_2,N) \to \Ext^i(M_1,N) \to \Ext^{i+1}(M_3,N) \to \dots
\tag{\ref{beta2}.4}
\end{equation}
We truncate (\ref{eq4}) and obtain the exact sequence
\begin{equation} \label{eq5}
\begin{split}
0 \to B_h \to \Ext^h(M_3,N) \to \Ext^h(M_2,N) \to \Ext^h(M_1,N) \to   \\
\dots \to \Ext^n(M_3,N) \to \Ext^n(M_2,N) \to \Ext^n(M_1,N) \to C_n \to 0,
\end{split}
\tag{\ref{beta2}.5}
\end{equation}
where $n$ and $h$ are integers such that $n > h > \f(M_j,N)$ for each $j$.
Taking the alternating sum of the lengths of the modules in (\ref{eq5}), we obtain
\begin{equation} \label{eq6}
g^{R}_{M_1,N}(h,n) - g^{R}_{M_2,N}(h,n) + g^{R}_{M_3,N}(h,n) = \pm \len_{R}(B_h) \pm \len_{R}(C_n).
\tag{\ref{beta2}.6}
\end{equation}
Since $e\geq \cx_{R}(M_1,N)$, it follows from (\ref{eq5}) that $\len_{R}(C_n)\leq \beta_n^R(M_{1},N)\leq A \cdot n^{e-1}$ for some real number $A$ and all $n\gg 0$. As $h$ is fixed, (\ref{eq3}) and (\ref{eq6}) give the equality we seek:
\begin{equation*}
\begin{split}
\h_e^R(M_1,N)-\h_e^R(M_2,N) + \h_e^R(M_3,N) & =
\lim_{n\to\infty} \frac{g_{M_{1},N}^R(h,n)}{n^e}-
\lim_{n\to\infty} \frac{g_{M_{2},N}^R(h,n)}{n^e}+
\lim_{n\to\infty} \frac{g_{M_{3},N}^R(h,n)}{n^e}\\
&=\lim_{n\to\infty} \frac{g^{R}_{M_1,N}(h,n) - g^{R}_{M_2,N}(h,n) + g^{R}_{M_3,N}(h,n)}{n^{e}} \\
& =\lim_{n\to\infty} \frac{\pm \len_{R}(B_h) \pm \len_{R}(C_n) }{n^{e}}=0.
\end{split}
\end{equation*}
Suppose now $e=0$ and $\len_{R}(M\tensor_RN)<\infty$. Then $\cx_{R}(M_{j},N)=0$, that is, $\Ext^i_R(M_j,N)=0$ for all $i\gg 0$ and all $j$. Moreover $\len_{R}(\Ext_{R}^i(M,N))<\infty$ for all $i$ and $j$. Therefore, taking the alternating sum of the lengths of $\Ext_{R}^i(M,N)$ in (\ref{eq4}), we conclude that $\h_0^R(M_2,N) = \h_0^R(M_1,N) + \h_0^R(M_3,N).$\\\\
(3) Write $R=R'/(x)$, where $R'=Q/(f_1,...,f_{r-1})$ and $x=f_{r}$. Then Proposition \ref{longexact} gives the following long exact sequence:
\begin{equation} \label{eq7}
... \to \Ext^{i}_R(M,N) \to \Ext_{R'}^i(M,N) \to \Ext^{i-1}_R(M,N) \to \Ext_R^{i+1}(M,N) \to...
\tag{\ref{beta2}.7}
\end{equation}
Suppose now $e=1$ and $\len_{R}(M\tensor_RN)<\infty$. Since $\Ext^{i}_{R'}(M,N)=0$ for all $i\gg 0$, it follows from the exact sequence (\ref{eq7})
that $\displaystyle{ 2 \cdot \h_1^{R}(M,N) = \h_{0}^{R'}(M,N)}$. Assume now $e\geq 2$. We truncate (\ref{eq7}) and obtain the exact sequence
\begin{equation}\label{eq8}
\begin{split}
0 \to B_h \to \Ext^{h}_R(M,N) \to \Ext^{h}_{R'}(M,N) \to \Ext^{h-1}_R(M,N)
\to  \dots                                                      \\
\to \Ext^{n}_R(M,N) \to \Ext^{n}_{R'}(M,N) \to \Ext^{n-1}_R(M,N) \to \Ext^{n+1}_R(M,N) \to C_n \to 0,
\end{split}
\tag{\ref{beta2}.8}
\end{equation}
where $n$ and $h$ are integers such that $n>h>\f(M,N)$.
Taking the alternating sum of the lengths of the modules in (\ref{eq8}) we get:
\begin{equation}\label{eq9}
g_{M,N}^{R'}(h,n) = (-1)^{n}\beta_n^R(M,N) -(-1)^{n}\beta_{n+1}^{R}(M,N) \pm  \beta^{R}_{h-1}(M,N) \pm \len_{R}(B_h) \pm \len_{R}(C_n).
\tag{\ref{beta2}.9}
\end{equation}
Since $C_{n}$ is a submodule of $\Ext^{R'}_{n+1}(M,N)$ and $\cx_{R'}(M,N)\leq e-1$, $\len_{R}(C_n)\leq \beta^{R'}_{n+1}(M,N)\leq B\cdot n^{e-2}$ for some real number $B$ and all $n\gg 0$. Therefore, using (\ref{eq9}) and the equality in Proposition \ref{beta1}(2), we conclude that
\begin{equation}\label{eq10}
g_{M,N}^{R'}(h,n) = \frac{2 \cdot n^{d-1} \cdot n_0}{(d-1)!}+ f(n) \textnormal{ for all } n\gg 0,
\tag{\ref{beta2}.10}
\end{equation}
where $f(t)\in \ZZ[t]$ is a polynomial of order at most $e-2$. Now (\ref{eq10}) implies
\begin{equation}\label{eq11}
\h^{R'}_{e-1}(M,N)= \lim_{n\to\infty}\frac{g_{M,N}^{R'}(h,n)}{n^{e-1}} = 2d \cdot \lim_{n\to\infty} \frac{n_0}{d!}n^{d-e}
\tag{\ref{beta2}.11}
\end{equation}
Notice, if $d\neq e$, $\displaystyle{\lim_{n\to\infty} \frac{n_0}{d!}n^{d-e}=0}$. Thus (\ref{eq11}) shows that $\displaystyle{\h^{R'}_{e-1}(M,N)= 2e \cdot \frac{n_0}{d!}n^{d-e}}.$ Therefore (\ref{eq3}) gives the equality we seek: $\displaystyle{\h^{R'}_{e-1}(M,N)=2e \cdot \h_{e}^R(M,N)}$.
\end{proof}

\section{Vanishing Results}\label{mainSection}

In this section we prove various vanishing results for $\Ext^i_R(M,N)$ when $M$ has finite complete intersection dimension. Our main tool will be the function $\h^{R}_{\bullet}(M,N)$.

\begin{rmk}\label{alfa1} Let $(R, \mathfrak m)$ be a local ring, and let $M$ and $N$ be finitely generated $R$-modules such that $\CI_{R}(M)<\infty$. Assume that $\f(M,N)<\infty$ and $e$ is an integer such that $e\geq \cx_{R}(M,N)$. Then $\h_e^R(M,N)$ is finite. One can see this as follows: Since $\CI_{R}(M)<\infty$, there exists a quasi-deformation $R \rightarrow S \stackrel{\alpha}{\twoheadleftarrow} P$ such that
$\pd_{P}(M\otimes_{R}S)<\infty$. Let $p$ be a minimal prime of $S/ \mathfrak mS$ and set $q=\alpha^{-1}(p)$. Then
the localized diagram $R \rightarrow S_{p} \twoheadleftarrow P_{q}$ is a quasi-deformation with zero-dimensional closed fiber such that $\pd_{P_{q}}(M\otimes_{R}S_{p})<\infty$ (cf. for example the proof of \cite[2.11]{SW}). Thus we may replace the original quasi-deformation with the localized one and assume that the closed fiber $S/ \mathfrak mS$ is Artinian. Therefore, since $R \rightarrow S$ is flat, $\f(M,N)=\textnormal{f}^{S}_{\textnormal{ext}}(M\otimes_{R}S,N\otimes_{R}S)$ and $\cx_{R}(M,N)=\cx_{S}(M\otimes_{R}S,N\otimes_{R}S)$. Write $S=P/(\underline{f})$ for some regular sequence
$\underline{f}$ of $P$. Then it follows from Theorem \ref{beta2}(1) that $\h^{S}_{e}(M\otimes_{R}S,N\otimes_{R}S)$ is finite. One can now define $\h_e^R(M,N)$ as in Definition \ref{function}; it is a multiple of $\h^{S}_{e}(M\otimes_{R}S,N\otimes_{R}S)$ and hence is finite.
\end{rmk}

\begin{thm} \label{alfa2} Let $R$ be a local ring, and let $M$ and $N$ be finitely generated $R$-modules. Assume $\CI_{R}(M)<\infty$ and $\f(M,N)<\infty$. Let $e$ be an integer such that $e\geq \cx_{R}(M,N)$. Assume further that $\h^{R}_{e}(M,N)=0$. If $\Ext^{n}_{R}(M,N)=\dots=\Ext^{n+e-1}_{R}(M,N)=0$ for some $n>\depth(R)-\depth(M)$, then $\Ext^{i}_{R}(M,N)=0$ for all $i>\depth(R)-\depth(M)$.
\end{thm}

\begin{proof} Set $c=\cx_{R}(M,N)$. If $c=0$, then $\Ext^{i}_{R}(M,N)=0$ for all $i\gg 0$, and
hence the result follows from \cite[4.2]{ArY} (cf. also \cite[4.7]{AvBu}). So we assume that $c\geq 1$.
Since $\CI_{R}(M)<\infty$, by Remark \ref{alfa1} and \cite[1.14]{AGP}, one can choose a quasi-deformation $R \rightarrow S \twoheadleftarrow P$ such that $\pd_{P}(M\otimes_{R}S)<\infty$, $S/ \mathfrak mS$ is Artinian and $P$ has infinite residue field. Note that $\h^{R}_{e}(M,N)=0$ if and only if $\h^{S}_{e}(M\otimes_{R}S,N\otimes_{R}S)=0$. Therefore we may assume $R=S$. One can now apply Lemma \ref{factor} to construct the rings $Q$ and $R'$.

We shall proceed by induction on $e$. We already settle the case $c=0$ or $e=0$. So suppose $e=1$. Then $c=1$, $R'=Q$ and $R=Q/(x)$. Since $\Ext^{i}_{Q}(M,N)=0$ for all $i\gg 0$, \cite[4.2]{ArY} shows that $\Ext^{i}_{Q}(M,N)=0$ for all $i>\depth(Q)-\depth(M)$. Thus Proposition \ref{1.2} implies that $\Ext^{i}_{R}(M,N)\cong \Ext^{i+2}_{R}(M,N)$ for all $i>\depth(R)-\depth(M)$. Set $w=\depth(R)-\depth(M)+1$. Since we assume $\f(M,N)<\infty$ there exist integers $a$ and $b$ such that $\beta_{w+2i}^{R}(M,N)=a$ and $\beta_{w+2i+1}^{R}(M,N)=b$ for all $i\geq 0$. Now, since $\h^{R}_{1}(M,N)=0$, we have:
$$\displaystyle{ \lim_{n\to\infty} \frac{(-1)^{w}\cdot a+(-1)^{w+1}\cdot b+(-1)^{w+2}\cdot a+(-1)^{w+3}\cdot b+\dots+(-1)^n\beta_n^R(M,N)}{n}=0} $$
The limit on the left is $(-1)^w(a-b)/2$, so $\beta_{i}^{R}(M,N)=\beta_{i+1}^{R}(M,N)$ for all $i\geq w$. Since $\Ext^{j}_{R}(M,N)=0$ for some integer $j\geq w$, we conclude that $\Ext^{i}_{R}(M,N)=0$ for all $i\geq w$, which is what we want. Assume now $e\geq 2$. Then Theorem \ref{beta2}(3) shows that $\h^{R'}_{e-1}(M,N)=0$. Moreover, by Proposition \ref{1.2}, we have that $\Ext^{n+1}_{R'}(M,N)=\dots=\Ext^{n+e-1}_{R'}(M,N)=0$. Since $\cx_{R'}(M,N)=\cx_{R}(M,N)-1\leq e-1$, the induction hypothesis and \cite[4.2]{ArY} imply that $\Ext^{i}_{R'}(M,N)=0$ for all $i>\depth(R')-\depth(M)$. Hence, using Proposition \ref{1.2} and the fact that $\Ext^{n}_{R}(M,N)=\Ext^{n+1}_{R}(M,N)=0$, we conclude $\Ext^{i}_{R}(M,N)=0$ for all $i\gg 0$. Thus $c=0$ and hence the result follows.
\end{proof}

Let $M$ and $N$ be finitely generated modules over a local ring $R$ such that $\CI_{R}(M)<\infty$. Set $c=\cx_{R}(M)$. If $\Ext^{n}_{R}(M,N)=\dots=\Ext^{n+c}_{R}(M,N)=0$ for some $n>\depth(R)-\depth(M)$, then it follows from \cite[2.6(1)]{Jo2} that $\Ext^{i}_{R}(M,N)$ for all $i>\depth(R)-\depth(M)$ (cf. also \cite[4.7]{AvBu}). Since $\cx_R(M,N)\leq \cx_{R}(M)$, Corollary \ref{alfa3} generalizes this result.

\begin{cor} \label{alfa3} Let $R$ be a local ring, and let $M$ and $N$ be finitely generated $R$-modules. Assume $\CI_{R}(M)<\infty$. Let $e$ be an integer such that $e>\cx_R(M,N)$. If $\Ext^{n}_{R}(M,N)=\dots=\Ext^{n+e-1}_{R}(M,N)=0$ for some $n> \depth(R)-\depth(M)$, then $\Ext^{i}_{R}(M,N)$ for all $i> \depth(R)-\depth(M)$.
\end{cor}

\begin{proof} If $\f(M,N)<\infty$, then Theorem \ref{beta2}(1) shows that $\h^{R}_{e}(M,N)=0$ and hence the result follows from Theorem \ref{alfa2}. Therefore it suffices to prove $\f(M,N)<\infty$. We shall proceed by induction on $\dim(R)$. There is nothing to prove if $\dim(R)=0$, since Theorem \ref{alfa2} applies directly. Thus assume $\dim(R)\geq 1$ and let $p$ be a non-maximal prime ideal of $R$. Since $\CI_{R_{p}}(M_{p})\leq \CI_{R}(M)$ \cite[1.6]{AGP} and $\cx_{R_{p}}(M_{p},N_{p})\leq \cx_{R}(M,N)$, the induction hypothesis implies that $\Ext^{i}_{R_{p}}(M_{p},N_{p})=0$ for all $i> \depth R_p$. Therefore $\f(M,N)<\infty$. This proves the claim.
\end{proof}

\begin{rmk}\label{alfa4} Avramov showed in \cite[9.3.7]{Av2} that the conclusion of Corollary \ref{alfa3} is not true in case $e=\cx_R(M,N)$; there are finitely generated modules $M$ and $N$ over a complete intersection ring $R$ such that $\Ext^{n}_{R}(M,N)=\dots=\Ext^{n+e-1}_{R}(M,N)=0$ for some $n>\depth(R)-\depth(M)$, where $e=\cx_{R}(M)=\cx_{R}(N)=\cx_{R}(M,N)>0$.
\end{rmk}

It was proved in \cite[5.1]{DaV} that over a Cohen-Macaulay ring $R$ that is an isolated singularity, one has $\cx_{R}(M,N)\leq \textnormal{min}\{\cx_{R}(M),\px_{R}(N)\}$. Therefore Proposition \ref{alfa3} gives the following result:

\begin{cor} Let $R$ be a local ring that is either Artinian or a one-dimensional domain. Let $M$ and $N$ be finitely generated $R$-modules. Assume $\CI_{R}(M)<\infty$ and $p=\px_{R}(N)\leq \cx_{R}(M)$. If $\Ext^{n}_{R}(M,N)=\dots=\Ext^{n+p}_{R}(M,N)=0$ for some $n> \depth(R)-\depth(M)$, then $\Ext^{i}_{R}(M,N)=0$ for all $i> \depth(R)-\depth(M)$.
\end{cor}

\begin{prop}\label{alfa5} Let $R$ be a local ring, and let $M$ and $N$ be finitely generated $R$-modules. Assume $\CI_{R}(M)<\infty$ and $\len_{R}(N) <\infty$. Let $e$ be an integer such that $e\geq \max\{1,\cx(M)\}$. Then $\h^R_e(M,N)=0$.
\end{prop}

\begin{proof} Note that, if $\cx(M)=0$, then the statement is obvious. Therefore we can assume $\cx(M)\geq 1$.
We only need to check the assertion for the case $N=k$, the residue field of $R$, as $N$ has a finite filtration by copies of $k$.
In view of Theorem \ref{beta2} and Lemma \ref{factor} we can assume $e=1$. Then $\cx(M)=1$. Moreover, by Lemma \ref{factor}, we can write $R=R'/(x)$ where $x$ is a non-zerodivisor of $R'$ and $\pd_{R'}M<\infty$. Now Theorem \ref{beta2}(3) shows that
$$ \h_1^{R}(M,k) = \h_{0}^{R'}(M,k) = \chi_{R'}(M)$$
where $\chi_{R'}(M)$ is the Euler characteristic of $M$ over $R'$. Since $x \in \Ann_{R'}(M)$, $\chi_{R'}(M)=0$ (cf. for example \cite[19.8]{Mat}).
\end{proof}

The following corollary now immediately follows from Theorem \ref{alfa2} and Proposition \ref{alfa5}.

\begin{cor} \label{alfa6} (\cite[3.5]{Be2}) Let $R$ be a local ring, and let $M$ and $N$ be finitely generated $R$-modules. Assume $\CI_{R}(M)<\infty$ and $\len_{R}(N) <\infty$. If $\Ext^{n}_{R}(M,N)=\dots=\Ext^{n+c-1}_{R}(M,N)=0$
for some $n> \depth(R)-\depth(M)$, where $c=\cx_{R}(M)$, then $\Ext^{i}_{R}(M,N)=0$ for all $i>\depth(R)-\depth(M)$.
\end{cor}

\begin{prop}\label{alfa7} Let $(R, \mathfrak m)$ be a local ring, and let $M$ and $N$ be finitely
generated $R$-modules. Assume the following conditions hold:
\begin{enumerate}
\item $\CI_{R}(M)<\infty$.
\item $\pd_{R_{p}}(M_{p})<\infty$ for all $p \in \textnormal{Spec}(R)-\{\mathfrak m\}$.
\item $[N]=0$ in $\overline{G}(R)_{\QQ}$.
\end{enumerate}
Set $c=\cx_{R}(M)$. If $\Ext^{n}_{R}(M,N)=\dots=\Ext^{n+c-1}_{R}(M,N)=0$
for some $n>\depth(R)-\depth(M)$, then $\Ext^{i}_{R}(M,N)=0$ for all $i>\depth(R)-\depth(M)$.
\end{prop}

\begin{proof} There is nothing to prove if $c=0$. So we may assume $c\geq 1$. Let $X$ be a finitely generated $R$-module. As $\pd_{R_{p}}(M_{p})<\infty$ for all $p \in \textnormal{Spec}(R)-\{\mathfrak m\}$, $\f(M,X)<\infty$. Hence Theorem \ref{beta2}(1)  shows that $\h^{R}_{c}(M,X)$ is finite. Therefore $\h^{R}_{c}(M,-):G(R)\rightarrow \QQ$ defines a linear map by Theorem \ref{beta2}(2). Note that, since $\CI_{R}(M)<\infty$, $\Ext^{i}_{R}(M,R)=0$ for all $i\gg 0$ (cf. \cite[1.4]{AGP} and \cite[ch.3]{AuBr}; see also \cite[1.2.7]{Chr}). Thus one obtains an induced map $\h^{R}_{c}(M,-):\overline{G}(R)_{\QQ}\rightarrow \QQ$. This implies, since $[N]=0$ in $\overline{G}(R)_{\QQ}$, that $\h^{R}_{c}(M,N)=0$. The result now follows from Theorem \ref{alfa2}.
\end{proof}

The next two results follow from Proposition \ref{Groth} and Proposition \ref{alfa7}. They improve \cite[2.6(1)]{Jo2} for finitely generated modules over certain local rings (see also Remark \ref{Grothrmk} for Corollary \ref{cor1}).

\begin{cor}\label{cor1} Let $R$ be a two-dimensional local normal domain such that the class group of $R$ is torsion. Let $M$ and $N$ be finitely generated $R$-modules. Assume that $\CI_{R}(M)<\infty$. Set $c=\cx_{R}(M)$. If $\Ext^{n}_{R}(M,N)=\dots=\Ext^{n+c-1}_{R}(M,N)=0$ for some $n> 2-\depth(M)$, then $\Ext^{i}_{R}(M,N)=0$ for all $i> 2-\depth(M)$.
\end{cor}

As discussed in Remark \ref{Grothrmk}, the condition that $\Cl(R)$ is torsion is subtle and may depend of the characteristic
of $R$. However  this Lemma below gives an easy way to find rings with torsion class groups. 

\begin{lem}
Let $R \subseteq S$ be a finite extension  of normal domains. If $\Cl(S)$ is torsion, then so is $\Cl(R)$.   
\end{lem}

\begin{proof}
Let $K,L$ be the quotient field of $R$ and $S$ respectively and $n = \lfloor L:K\rfloor $. There are well-known maps: $i: \Cl(R) \to \Cl(S)$ and $j: \Cl(S) \to \Cl(R)$ such  that $j\circ i = n\dot \id_{\Cl(R)}$, see Chapter 7, Section 4.8 of \cite{Bour}. Let $\alpha$ be an element of $\Cl(R)$. Since $\Cl(S)$ is torsion, there is an integer $n_1$ such that $n_1i(\alpha)=0$ in $\Cl(S)$. Then it follows that $nn_1\alpha= 0$ in $\Cl(R)$. 
\end{proof}

\begin{eg}
The above Lemma shows, for example, that if $R$ is a Veronese subring of the polynomial rings $S=k[x,y]$, then $\Cl(R)$ is torsion. So local rings of $R$ would satisfy the condition of Corollary \ref{cor1}. One can use the Lemma repeatedly to generate similar examples.  
\end{eg}

\begin{cor}\label{alfa8} Let $R$ be a one-dimensional local domain, and let $M$ and $N$ be finitely generated $R$-modules. Assume $\CI_{R}(M)<\infty$. Set $c=\cx_{R}(M)$. If $\Ext^{n}_{R}(M,N)=\dots=\Ext^{n+c-1}_{R}(M,N)=0$
for some $n>1-\depth(M)$, then $\Ext^{i}_{R}(M,N)=0$ for all $i>1-\depth(M)$.
\end{cor}

When the ring considered is Gorenstein, one can improve Corollary \ref{alfa8} by using the fact that every finitely generated module over a Gorenstein ring has a maximal Cohen-Macaulay approximation \cite{AuBu} (cf. also \cite[chapter 9]{LW}).

\begin{prop}\label{alfa9} Let $R$ be a one-dimensional local Gorenstein domain, and let $M$ and $N$ be finitely generated $R$-modules. Assume $\CI_{R}(M)<\infty$ and $\Ext^{1}_{R}(M,N)=\dots=\Ext^{c}_{R}(M,N)=0$, where $c=\cx_{R}(M)\geq 1$. Then the following holds:
\begin{enumerate}
\item $M$ is torsion-free and $\Ext^{i}_{R}(M,N)=0$ for all $i\geq 1$.
\item $N$ is torsion-free if and only if $\Ext^{1}_{R}(N,M)=0$ if and only if $\Ext^{i}_{R}(N,M)=0$ for all $i\geq 1$.
\end{enumerate}
\end{prop}

\begin{proof} We will first prove that $M$ is torsion-free and $\Ext^{i}_{R}(M,N)=0$ for all $i\geq 1$. If $M$ is torsion-free, then the result follows from Corollary \ref{alfa8}.
Suppose now $M$ has torsion, i.e., $\depth(M)=0$. Then there exists an exact sequence
$$(\ref{alfa9}.1)\;\; 0 \rightarrow T \rightarrow X \rightarrow M \rightarrow 0$$ where $X$ is torsion-free and
$T$ has finite injective dimension \cite{AuBu}. Since $R$ is a one-dimensional Gorenstein ring and $\depth(M)=0$,
the depth lemma and $(\ref{alfa9}.1)$ imply that $T$ is free. Furthermore, by $(\ref{alfa9}.1)$, $\Ext^{1}_{R}(X,N)=\dots=\Ext^{c}_{R}(X,N)=0$ and $\Ext^{i}_{R}(M,N) \cong \Ext^{i}_{R}(X,N)$ for all $i\geq 2$.
Since $\CI_{R}(M)<\infty$, there exists a quasi-deformation $R \rightarrow S \twoheadleftarrow P$ such that $\pd_{P}(M\otimes_{R}S)<\infty$. Tensoring $(\ref{alfa9}.1)$ with $S$ over $R$, we see that $\pd_{P}(X\otimes_{R}S)<\infty$. Hence $\CI_{R}(X)<\infty$. Since $\cx_{R}(M)=\cx_{R}(X)$, it now follows from Corollary \ref{alfa8} that $\Ext^{i}_{R}(X,N)=0$ for all $i\geq 1$. Thus $\Ext^{i}_{R}(M,N)=0$ for all $i\geq 1$. However, since $\depth(M)=0$,
$\Ext^{1}_{R}(M,N)\neq 0$ (cf. \cite[4.2]{ArY} or \cite[4.8]{AvBu}). Therefore $M$ is torsion-free and $\Ext^{i}_{R}(M,N)=0$ for all $i\geq 1$. This proves (1).

Note that, since $\CI_{R}(M)<\infty$ and $\Ext^{i}_{R}(M,N)=0$ for all $i\geq 1$, \cite[4.7]{AvBu} shows that $\widehat{\Ext}^{n}_{R}(M,N)=0$ for all $n \in \ZZ$, where $\widehat{\Ext}^{n}_{R}(M,N)$ denotes the $n$th stable cohomology module. Therefore, as $M$ is torsion-free, $\widehat{\Tor}^{R}_{n}(M^{\ast},N)\cong \widehat{\Ext}^{-n-1}_{R}(M^{\ast\ast},N)\cong \widehat{\Ext}^{-n-1}_{R}(M,N)=0$ for all $n \in \ZZ$. In particular $\Tor^{R}_{i}(M^{\ast},N)=0$ for all $i\geq 1$ (cf. \cite[4.4.6 \& 4.4.7]{AvBu}). It was proved in \cite{CJ} that if $X$ and $Y$ are finitely generated modules over a local Gorenstein ring $A$ such that $\widehat{\Tor}^{A}_{n}(X,Y)=0$ for all $n \in \ZZ$ and $\Tor^{A}_{i}(X,Y)=0$ for all $i\geq 1$, then the depth formula holds, i.e., $\depth_{A}(X\otimes_{A}Y)=\depth_A(X)+\depth_A(Y)-\depth(A)$. This implies $\depth(M^{\ast}\otimes_{R}N)=\depth(N)$.
Thus $N$ is torsion-free if and only if $M^{\ast}\otimes_{R}N$ is torsion-free. Now, if $N$ is torsion-free, then it follows from
\cite[2.7]{Jo4} that $\Ext^{i}_{R}(N,M)=0$ for all $i\geq 1$. On the other hand, if $\Ext^{1}_{R}(N,M)=0$, then \cite[4.6]{HW} implies that $M^{\ast}\otimes_{R}N$ is torsion-free. This proves (2).
\end{proof}

As the vanishing of $\Ext^{i}_{R}(M,N)$ for all $i\gg 0$ over a hypersurface $R$ forces $M$ or $N$ to have finite projective dimension \cite[5.12]{AvBu}, Corollary \ref{alfa8} and Proposition \ref{alfa9} yield the following result:

\begin{cor}\label{alfa10} Let $R$ be a one-dimensional local hypersurface domain, and let $M$ and $N$ be finitely generated $R$-modules. If $\Ext^{n}_{R}(M,N)=0$ for some $n\geq 1$, then $\pd_{R}M<\infty$ or $\pd_{R}N <\infty$. In particular, if $\Ext^{1}_{R}(M,M)=0$, then $M$ is free.
\end{cor}

\begin{rmk} We note that the conclusion of Corollary \ref{alfa10} is not true over an arbitrary hypersurface; see, for example, \cite[4.3]{AvBu}. We shall also note if $M$ is a finitely generated torsion-free module over a one-dimensional complete intersection domain $R$, it is not known whether $\Ext^{1}_{R}(M,M)=0$ forces $M$ to be free (cf. \cite[page 473]{HW}).
For modules of bounded Betti numbers, we have the following result (cf. also \cite[4.17]{Ce} and \cite[5.5]{Da1}).
\end{rmk}

\begin{prop}\label{alfa11} Let $R$ be a local ring, and let $M$ be finitely generated $R$-module. Assume $\CI_{R}(M)<\infty$ and $M$ has bounded Betti numbers. Assume further that $[M]=0$ in $\overline{G}(R)_{\QQ}$. If $\Ext^{n}_{R}(M,M)=0$ for some $n>\depth(R)-\depth(M)$, then $\pd_{R}(M)<n$.
\end{prop}

\begin{proof} Notice that $\cx_{R}(M) \leq 1$. We proceed by induction on $\dim(R)$. Assume first $\dim(R)=0$. Then Corollary
\ref{alfa6} and \cite[4.2]{AvBu} imply that $\pd_{R}(M)<\infty$. Suppose now $\dim(R)\geq 1$. Then the induction hypothesis implies that $\pd_{R_{p}}(M_{p})<\infty$ for all non-maximal prime ideals $p$ of $R$. Hence Proposition \ref{alfa7} and \cite[4.2]{AvBu} give the desired result.
\end{proof}

\section*{Acknowledgments}
We would like to thank Luchezar Avramov, Dale Cutkosky, Mark Walker and Roger Wiegand for their valuable comments and suggestions during the preparation of this paper.

\end{document}